\documentclass[11pt,reqno]{amsart}
\usepackage{latexsym}
\RequirePackage{amsmath}   \RequirePackage{amssymb}
\RequirePackage{amsthm}   \RequirePackage{epsfig}
\theoremstyle{plain}
\newtheorem{theorem}{Theorem}

\newtheorem{lemma}{Lemma} %[theorem]
\newtheorem{proposition}[lemma]{Proposition} %[theorem]
\newtheorem*{corollary*}{Corollary}
\newtheorem{theoremO}{Theorem}
\newtheorem{lemmO}[theoremO]{Lemma}

\newtheorem*{conjecture*}{Conjecture}

\theoremstyle{definition}

\theoremstyle{remark}

\newcommand{\SC}{{\mathbb C}}  \newcommand{\SD}{{\mathbb D}}  
\newcommand {\SR}{{\mathbb R}}  \newcommand{\ST}{{\mathbb T}}  
\newcommand{\al}{\alpha}  \newcommand{\ga}{\gamma} \newcommand{\de}{\delta} 
    \newcommand{\ze}{\zeta}
\def\t{\theta}    \newcommand{\la}{\lambda}
\newcommand{\si}{\sigma}  \newcommand{\vp}{\varphi}  \newcommand{\om}{\omega}
\newcommand{\De}{\Delta}  \newcommand{\Om}{\Omega}  
\newcommand{\cA}{{\mathcal A}}   
 \newcommand{\cF}{{\mathcal F}}

\newcommand{\be}{\begin{equation}}
\newcommand{\ee}{\end{equation}}
\newcommand{\bea}{\begin{eqnarray}}
\newcommand{\eea}{\end{eqnarray}}

\def\le{\left}
\def\ri{\right}

\begin{document}

\title[Criteria for univalence and quasiconformal extension]{Criteria for univalence and quasiconformal extension for harmonic mappings on planar domains}  

\author[I. Efraimidis]{Iason Efraimidis}

\address{Department of Mathematics and Statistics, Texas Tech University, Box 41042, Lubbock, TX 79409, United States} \email{iason.efraimidis@ttu.edu}  

\maketitle

\begin{abstract}
If $\Om$ is a simply connected domain in $\overline{\SC}$ then, according to the Ahlfors-Gehring theorem, $\Om$ is a quasidisk if and only if there exists a sufficient condition for the univalence of holomorphic functions in $\Omega$ in relation to the growth of their Schwarzian derivative. We extend this theorem to harmonic mappings by proving a univalence criterion on quasidisks. We also show that the mappings satisfying this criterion admit a homeomorphic extension to $\overline{\SC}$ and, under the additional assumption of quasiconformality in $\Om$, they admit a quasiconformal extension to $\overline{\SC}$. 

The Ahlfors-Gehring theorem has been extended to finitely connected domains $\Om$ by Osgood, Beardon and Gehring, who showed that a Schwarzian criterion for univalence holds in $\Om$ if and only if the components of $\partial\Om$ are either points or quasicircles. We generalize this theorem to harmonic mappings.
\end{abstract}

\section{Introduction}  
%%%%%%%%%%%%%%%%%%%%%%%%%%%%%%%%%%%%%%%%%%%%%%%%%%
%%%%%%%%%%%%%%%%%%%%%%%%%%%%%%%%%%%%%%%%%%%%%%%%%%

\subsection{Schwarzian derivative} For a locally univalent analytic function $f$ the Schwarzian derivative is defined by
$$ 
Sf \, = \, \left( f'' /f' \right)' -\tfrac{1}{2} \left( f'' /f'  \right)^2. 
$$
This operator vanishes identically if and only if $f$ is a M\"obius transformation. According to a classical theorem of Nehari the bound 
\be \label{Neh-t-AhlWe}
|Sf(z)| \leq \frac{2\,t}{(1-|z|^2)^2} \, , \qquad z\in\SD,
\ee
for $t=1$ implies the global univalence of $f$ in the unit disk $\SD$, while another classical result, proved by Ahlfors and Weill, gives an explicit quasiconformal extension of $f$ to $\overline{\SC}$ under the assumption that $f$ satisfies \eqref{Neh-t-AhlWe} with $t<1$. 

Let $\Om$ be a hyperbolic domain in $\overline{\SC}$, meaning that it has at least three boundary points, and let $\pi:\SD\to\Om$ be a universal covering map. Then the density $\la_\Om$ of the hyperbolic (Poincar\'e) metric in $\Om$ is defined by 
$$
\la_\Om\big(\pi(z)\big) |\pi'(z)| \, = \, \la_\SD(z)\, = \,\frac{1}{1-|z|^2}, \qquad z\in \SD, 
$$
and is independent of the choice for the covering $\pi$. The size of the Schwarzian derivative of a locally univalent holomorphic function $f:\Om\to\SC$ is measured by its norm, given by 
\be \label{def-Schw-norm}
\|Sf\|_\Om \, = \, \sup_{z\in\Om} \, \la_\Om(z)^{-2} |Sf(z)|.
\ee 
The inner radius of $\Om$ is defined as the number
$$
\si(\Om) \, = \, \sup \{\, c\geq0 \, : \, \, \|Sf\|_\Om \leq c \,\Rightarrow \, f \, \text{univalent} \}.  
$$
This domain constant is M\"obius invariant. We say that a univalence criterion holds in $\Om$ if and only if $\sigma(\Om)>0$. We have that $\si(\SD)=2$ since, as shown by Hille, the constant 2 in Nehari's theorem is sharp. Lehtinen showed that every simply connected domain $\Om$ satisfies $\si(\Om) \leq 2$, with equality only in the case when $\Om$ is a disk or a half-plane. See Lehto's book \cite[Ch.III, \S 5]{Leh} for more information on the inner radius.
 
A domain $\Om$ is a quasidisk if it is the image of $\SD$ under a quasiconformal self-map of $\overline{\SC}$. The boundary of a quasidisk is called a quasicircle. 

Our starting point is the following theorem by Ahlfors and Gehring. 

\begin{theoremO}[\cite{Ah,Ge}] \label{thm-inner-radius-Ah-Ge}
Let $\Om$ be a simply connected domain in $\overline{\SC}$. Then $\si(\Om) >0$ if and only if $\Om$ is a quasidisk. Moreover, any holomorphic function $f:\Om\to\SC$ that satisfies $\|Sf\|_\Om <\si(\Om)$ admits a quasiconformal extension to $\overline{\SC}$. 
\end{theoremO}

Ahlfors \cite{Ah} proved the quasiconformal extension criterion on a quasidisk stated here; the univalence criterion follows from it. The fact that such a criterion can only occur on a quasidisk was shown by Gehring \cite{Ge}. This has been generalized to a class of multiply connected domains with the following theorem.

\begin{theoremO}[\cite{BG,O80}]  \label{thm-Ge-Os}
Let $\Om$ be a finitely connected domain in $\overline{\SC}$. Then $\si(\Om) >0$ if and only if every boundary component of $\Om$ is either a point or a quasicircle.
\end{theoremO}

The univalence criterion was proved by Osgood \cite{O80}, while the fact that the class of domains where it holds cannot be enlarged was proved by Beardon and Gehring \cite{BG}. We note that for finitely connected domains the property of $\partial\Om$ described in this theorem characterizes the class of uniform domains, introduced by Martio and Sarvas \cite{MS79} (see also \cite[\S3.5]{GeHa}). We also mention that there exist both examples of infinitely connected circle domains with positive (see \cite{O80}) and zero (see \cite{BG}) inner radius. 

In recent years there is much activity in extending the theory of the Schwarzian derivative to harmonic mappings, to which we now turn.

%%%%%%%%%%%%%%%%%%%%%%%%%%%%%%%%%%
\subsection{Harmonic mappings} A complex-valued harmonic mapping $f$ in a simply connected domain $\Omega\subset\SC$ has a canonical decomposition $f=h+\overline{g}$, where $h$ and $g$ are analytic in $\Om$. The mapping $f$ is locally univalent if and only if its Jacobian $J_f=|h'|^2-|g'|^2$ does not vanish, and is said to be orientation-preserving if its dilatation $\omega=g'/h'$ satisfies $|\omega|<1$ in $\Om$. We say that $f$ is normalized if $h(z_0)=g(z_0)=0$ and $h'(z_0)=1$ for some specified $z_0\in\Omega$.

The Schwarzian derivative has been extended to harmonic mappings by two complementary definitions: a first one appeared in \cite{CDO03} and another was later introduced by Hern\'andez and Mart\'in in \cite{HM15}. We will follow the latter, which seems to be better suited when one does not wish to consider the Weierstarss-Enneper lift to a minimal surface. Hence, the Schwarzian derivative of a locally univalent harmonic mapping $f$ is defined by 
$$
S_f \, = \, \rho_{zz} - \tfrac{1}{2} (\rho_z)^2, \qquad \rho =\log J_f,
$$
where $J_f=|f_z|^2-|f_{\overline{z}}|^2$ is the Jacobian. If $\Om$ is simply connected and, therefore, the decomposition $f=h+\overline{g}$ is valid, then the above takes the form  
\be \label{Sch-HM}
S_f \, = \, Sh + \frac{\overline{\om}}{1-|\om|^2} \left( \frac{h''}{h'} \om'-\om'' \right) - \frac{3}{2} \left( \frac{\om' \overline{\om}}{1-|\om|^2} \right)^2. 
\ee
Note that we are using the notation $Sf$ when we know that the mapping $f$ is holomorphic and the notation $S_f$, with the mapping $f$ as a subscript, in the more general setting of harmonic mappings. This operator satisfies the chain rule, for if $\vp$ is analytic in $\Om$ and such that the composition $f\circ \vp$ is well defined then we have that  
\be \label{eq-chain}
S_{f\circ \vp} \, = \, S_f\circ \vp \, (\vp')^2 + S\vp. 
\ee
The Schwarzian norm of a harmonic mapping $f$ is defined exactly as in \eqref{def-Schw-norm}. It was shown in \cite{HM15} that $\|S_f\|_\Om=0$ implies that $f$ is an affine map of a M\"obius transformation and is, therefore, univalent.

%%%%%%%%%%%%%%%%%%%%%%%%%%%%%%%%%%
\subsection{Main results} We define the \emph{harmonic inner radius} of a hyperbolic domain $\Om$ in $\overline{\SC}$ as the constant 
$$
\si_H(\Om) \, = \, \sup \{ \, t \geq0 \, : \, f \; \text{harmonic in}\; \Om \; \text{with}  \; \|S_f\|_\Om\leq t \, \Rightarrow \, f \, \text{univalent} \}. 
$$
Evidently,  
\be \label{harm-inn-less-inn}
\si_H(\Om) \leq \si(\Om),
\ee
since every holomorphic function is harmonic. This shows that if $\si_H(\Om)>0$ and $\Om$ is finitely connected then, in view of Theorem~\ref{thm-Ge-Os}, every boundary component of $\Om$ is either a point or a quasicircle. For $\Om=\SD$ it was shown in \cite{HM15-2} that $\si_H(\SD)>0$. We prove that the harmonic inner radius is positive for all quasidisks.  

\begin{theorem}\label{thm-harmonic-in-rad}
Let $\Om$ be a quasidisk. Then there exists a constant $c>0$, depending only on $\si(\Om)$, such that if $f$ is harmonic in $\Om$ with $\|S_f\|_\Om \leq c$ then $f$ is univalent in $\Om$ and admits a homeomorphic extension to $\overline{\SC}$. 
\end{theorem}
 
The proof of the univalence criterion in Theorem~\ref{thm-harmonic-in-rad} follows closely the reasoning in \cite{HM15-2}: We show that if a harmonic mapping $f=h+\overline{g}$ has small Schwarzian derivative then so does its analytic part $h$, and therefore $h$ is univalent by Theorem~\ref{thm-inner-radius-Ah-Ge}. The same can then be said about $h+ag$, the analytic part of the affine transformation $f+a\overline{f}, a\in\SD$. Finally, Hurwitz' theorem shows that $h+ag$ is univalent for every $a\in\overline{\SD}$ and by an elementary rotational argument we get that $f$ is injective. The crucial step in the generalization from $\SD$ to a quasidisk involves the hyperbolic derivative of admissible dilatations and is given in Lemma~\ref{lem-R-lambda-Om}. The homeomorphic extension under these hypotheses is novel even for the unit disk. 

Further, for more general domains we prove the following theorem.

\pagebreak 

\begin{theorem} \label{thm-finit-conn}
Let $\Om$ be a finitely connected domain. The following are equivalent. 
\begin{itemize}
\item[(i)] Every boundary component of $\Om$ is either a point or a quasicircle. 
\item[(ii)] $\si(\Om) >0$
\item[(iii)] $\si_H(\Om) >0$
\end{itemize}
\end{theorem}

The equivalence of (i) and (ii) was given in Theorem~\ref{thm-Ge-Os}, while the direction (iii)$\Rightarrow$(ii) follows trivially from \eqref{harm-inn-less-inn}. We prove the direction (i)$\Rightarrow$(iii) in Section~\ref{sect-fin-conn} by using Osgood's \cite{O80} quasiconformal decomposition of $\Om$ and the homeomorphic extension of Theorem~\ref{thm-harmonic-in-rad}. 

Finally, we give sufficient conditions for a harmonic mapping defined on a quasidisk to admit a quasiconformal extension to $\overline{\SC}$. 

\begin{theorem}\label{thm-qc-extension}
Let $\Om$ be a quasidisk and let $d\in[0,1)$. Then there exists a constant $c>0$, depending only on $d$ and on $\si(\Om)$, such that if $f$ is harmonic in $\Om$ with $\|S_f\|_\Om \leq c$ and its dilatation satisfies $\sup_{z\in\Om}|\om(z)| \leq d $ then $f$ admits a quasiconformal extension to $\overline{\SC}$. 
\end{theorem}

If $d=0$ then $f$ is analytic and we recover Ahlfors' \cite{Ah} theorem. For our proof of Theorem~\ref{thm-qc-extension} we consider the dilation $\Om_r$ (\emph{i.e.}~the image of $|z|< r$, for $r<1$, under the Riemman mapping of $\Om$) and use quasiconformal reflections to obtain a $K$-quasiconformal extension of $f\big|_{\Om_r}$ to $\overline{\SC}$. The desired extension is then harvested as the limit for $r\to1$, once we prove that $K$ is independent of $r$ by studying the cross-ratio of points on the image of $\partial\Om_r$ under $f$. For the latter we use the insightful ideas of \cite{HM13}. We prove Theorem~\ref{thm-qc-extension} in Section~\ref{sect-qc-ext}. 

We note that for the case when $\Om$ is the unit disk $\SD$, Theorem~\ref{thm-qc-extension} was proved in \cite{HM15-2}. However, the slightly stronger statement made there, namely that the constants $c$ and $d$ in the hypotheses are independent, does not seem to follow from the suggestion of the authors to argue as in \cite{HM13} for a proof. It is interesting to ask if this stronger statement can be rigorously proved.

\section{Preliminaries} % \label{sect-prelim}
%%%%%%%%%%%%%%%%%%%%%%%%%%%%%%%%%%%%%%%%%%%%%%%%%%
%%%%%%%%%%%%%%%%%%%%%%%%%%%%%%%%%%%%%%%%%%%%%%%%%%
 
\subsection{Bounded Schwarzian derivative} A well-known theorem of \mbox{Pommerenke} \cite{Po64} states that if $\vp$ is analytic and locally univalent in $\SD$ then 
\be  \label{Pom-lem}
(1-|z|^2)\left| \frac{\vp''(z)}{\vp'(z)}  -\frac{2 \, \bar{z}}{1-|z|^2} \right| \, \leq \, 2 \sqrt{1+\frac{\|S\vp\|_{\SD}}{2}}, \qquad z\in\SD. 
\ee
From this, a simple application of Montel's theorem shows that the set of functions 
$$
\{\, \vp  \; : \; \|S\vp\|_{\SD} \leq c \,\},
$$
where $c>0$, constitutes a normal family (the set of fucntions $(\log\vp')'$ is locally uniformly bounded in $\SD$, and so is the set of functions $\vp$). 

According to Theorem 6 in \cite{HM15}, if $f=h+\overline{g}$ is harmonic and locally univalent in $\SD$ then 
\be \label{Sf-bdd-iff-Sh}
\|S_f\|_\SD<\infty \qquad \text{if and only if} \qquad \|Sh\|_\SD<\infty. 
\ee

\vskip.2cm
%%%%%%%%%%%%%%%%%%%%%%%%%%%%%%%%%%
\subsection{Nomralizations for harmonic mappings} Let $\Om$ be a simply connected domain that contains the origin and let $t\geq 0$. We then denote by $\cF_t(\Om)$ the set of all sense-preserving harmonic mappings $f=h+\overline{g}$ in $\Om$ which satisfy $\|S_f\|_\Om\leq t$ and are normalized by $h(0)=g(0)=0$ and $h'(0)=1$. Let 
$$
\cF_t^0(\Om)   \, = \, \{ f \in  \cF_t(\Om) \; : \; g'(0)=0 \}. 
$$
The following proposition is well known among experts but we will include a proof here for the convenience of the reader. We note that the compactness of $\cF_t^0(\SD)$ was shown in \cite{CHM}. 

%%%%%%%%%%%%%%%%%
\begin{proposition}\label{normality}
The family $\cF_t(\Om)$ is normal. The family $\cF_t^0(\Om)$ is normal and compact. 
\end{proposition}

\begin{proof}To show that $\cF_t^0(\SD)$ is normal we observe that the set
$$
\{\, h   \; : \; f=h+\overline{g}\in\cF_t^0(\SD) \, \}
$$
is a normal family in view of inequality~\eqref{Pom-lem} and since $\|Sh\|_\SD$ is bounded by \eqref{Sf-bdd-iff-Sh} (even more so, a close inspection of  \cite[Thm.6]{HM15} shows that it is uniformly bounded). Also, again by \eqref{Pom-lem}, the functions $h'$ are locally uniformly bounded in $\SD$ and, in view of the condition $|g'|<|h'|$, so are the functions $g'$. Hence the family $\{ g :  f\in\cF_t^0(\SD)  \}$ is also normal. Thus $\cF_t^0(\SD)$ is a normal family. 

For a simply connected domain $\Om$, with $0\in\Om$, let $f\in\cF_t^0(\Om)$ and consider the mapping $F=\vp'(0)^{-1}f\circ \vp$, where $\vp$ is a Riemann map for which $\Om=\vp(\SD)$ and $\vp(0)=0$. By the chain rule \eqref{eq-chain} we have that 
$$
S_F(z) \, = \, S_f\big( \vp(z) \big) \vp'(z)^2 + S\vp(z), \qquad z\in \SD. 
$$
We then compute  
\be \label{Riemann-Schwarzian}
\frac{|S_F(z)-S\vp(z)|}{ \la_\SD(z)^2} \, = \, \frac{| S_f\big( \vp(z) \big)|}{ \la_\Om\big( \vp(z) \big)^2}, \qquad z\in \SD, 
\ee
which shows that $\|S_F-S\vp\|_\SD = \|S_f\|_\Om$. According to Kraus' theorem \cite[Ch.II, Thm.1.3]{Leh} we have that $\|S\vp\|_\SD \leq 6$. Therefore, we get that
$$
\|S_F\|_\SD \, \leq \, \|S_f\|_\Om  + \|S\vp\|_\SD \, \leq \, t+6.
$$
Hence the set $\{\vp'(0)^{-1} f\circ\vp : f\in\cF_t^0(\Om) \}$ is included in the normal family $\cF_{t+6}^0(\SD)$ and is, therefore, a normal family itself. The claim that $\cF_t^0(\Om)$ is normal follows directly from this. 

Any $f\in\cF_t(\Om)$ can be written as an affine transform of a mapping in $\cF_t^0(\Om)$, in particular, if $b_1=g'(0)$ we may write $f=f_0+\overline{b_1}\,\overline{f_0}$ for some $f_0\in\cF_t^0(\Om)$. Hence $|f|\leq 2 |f_0|$ and with the observation that Montel's criterion for normality remains valid for families of harmonic mappings (see \cite[p.80]{Du3}) we conclude that $\cF_t(\Om)$ is normal. 

Finally, the compactness of $\cF_t^0(\Om)$, as shown in the course of the proof of Theorem 3 in \cite{CHM}, amounts to the observation that if $f_n$ is a sequence in $\cF_t^0(\Om)$ that converges to $f$ locally uniformly in $\Om$ then $S_{f_n} \to S_f$ pointwise in $\Om$. Hence $f\in\cF_t^0(\Om)$ and so this class is compact. 

\end{proof}
%%%%%%%%%%%%%%%%%

%%%%%%%%%%%%%%%%%%%%%%%%%%%%%%%%%%
\subsection{Affine invariance}  Let $a\in\SD$ and consider the affine transformation of a mapping $f$ in $\cF_t(\Om)$ given by  
\be \label{aff-trans}
F(z) \, =\, A_a f(z) \, =\, \frac{f(z) + a\, \overline{f(z)} }{ 1+a g'(0)}, \qquad z\in\Om. 
\ee
Now $F\in\cF_t(\Om)$ since it satisfies $S_F\equiv S_f$ (see \cite[Prop.1]{HM15}) and the corresponding normalizations. It can easily be seen that the dilatation of $F$ is given by \linebreak $\om_F=\nu\vp_{\overline{a}}\circ\om$, where $\nu\in\ST(=\partial\SD)$ and
$$
\vp_a(z)=\frac{a+z}{1+\overline{a}z}, \qquad z\in\SD.
$$
If we make the choice $a=-\overline{\om(0)}$ then $F$ will have the additional normalization $\om_F(0)=\nu\vp_{\overline{a}}(-\overline{a}) =0$ and will therefore belong to $\cF_t^0(\Om)$.  

%%%%%%%%%%%%%%%%%%%%%%%%%%%%%%%%%%
\subsection{Quasiconformal mappings} A sense-preserving homeomorphism $f:\Om\to\SC$ is said to be $K$-quasiconformal, $K\geq1$, if it is absolutely continuous on lines and satisfies $|f_{\overline{z}}| \leq k |f_z|$, where $k=(K-1)/(K+1)$, almost everywhere in $\Om$. \linebreak A mapping is called quasiconformal if it is $K$-quasiconformal for some $K\geq1$. The 1-quasiconformal mappings are the conformal mappings. Note that a harmonic mapping is quasiconformal if its dilatation satisfies $|\om|\leq k<1$.

A domain $\Om$ is called a $K$-quasidisk, and its boundary a $K$-quasicircle, if it is the image of $\SD$ under a $K$-quasiconformal self-map of $\overline{\SC}$. According to a theorem of Ahlfors \cite{Ah} a Jordan curve $\ga\subset\overline{\SC}$ is a quasicircle if and only if for all points $z_j\in\ga, j=1,2,3,4$, such that $z_1$ and $z_3$ separate $z_2$ and $z_4$, the cross-ratio 
$$
(z_1,z_2,z_3,z_4) = \frac{(z_1-z_2)(z_3-z_4)}{(z_1-z_3)(z_2-z_4)}
$$
satisfies 
$$
|(z_1,z_2,z_3,z_4)| \leq C, 
$$
for some constant $C>0$. 

Let $\Om_1$ and $\Om_2$ be the complementary components of a Jordan curve $\ga\subset\overline{\SC}$. Then a sense-reversing homeomorphism $\la$ of the sphere onto itself is a reflection across $\ga$ if it maps $\Om_1$ onto $\Om_2$ and keeps every point on $\ga$ fixed. According to a theorem of K\"uhnau \cite{Ku88} (see also Theorem 2.1.4 in \cite{GeHa}) $\ga$ is a $K$-quasicircle if and only if it admits a reflection $\la$ such that $\la(\overline{z})$ is $K$-quasiconformal.

\section{Preparatory lemmas}
%%%%%%%%%%%%%%%%%%%%%%%%%%%%%%%%%%%%%%%%%%%%%%%%%%
%%%%%%%%%%%%%%%%%%%%%%%%%%%%%%%%%%%%%%%%%%%%%%%%%%
 
%%%%%%%%%%%%%%%%%%%%%%%%%%%%%%%%%%
\subsection{Adaptations to general domains} The following proposition is a straightforward generalization of \eqref{Sf-bdd-iff-Sh}. 

%%%%%%%%%%%%%%%%%
\begin{proposition} \label{general-Sf-iff-Sh}
Let $\Om$ be a simply connected domain and $f=h+\overline{g}$ a sense-preserving locally univalent harmonic mapping in $\Om$. Then $\|S_f\|_\Om<\infty$ if and only if $\|S_h\|_\Om<\infty$. 
\end{proposition}
\begin{proof}
Let $\vp$ be a Riemann map for which $\Om=\vp(\SD)$ and consider the mappings $F=f\circ \vp$ and $H=h\circ \vp$. By a calculation we saw in \eqref{Riemann-Schwarzian} we readily have that
$$
\|S_F-S\vp\|_\SD \, = \, \|S_f\|_\Om \qquad \text{and} \qquad \|SH-S\vp\|_\SD \, = \, \|Sh\|_\Om. 
$$
The proposition is a direct consequence of Theorem 6 in \cite{HM15} and Kraus' theorem $\|S\vp\|_\SD\leq6$.
\end{proof}
%%%%%%%%%%%%%%%%%

\vskip.2cm
We now generalize inequality \eqref{Pom-lem} to an arbitrary hyperbolic domain $\Om$. We write $d(z)=dist(z,\partial\Om)$ for the distance of a point $z$ in $\Om$ to the boundary $\partial\Om$.

\begin{proposition} \label{general-Pommerenke}
If $h$ is analytic and locally univalent in $\Om$ then  
$$
\left| \frac{h''(z)}{h'(z)} \right| \, \leq \, \frac{2}{d(z)} \sqrt{1+\frac{\|Sh\|_\Om}{2}}, \qquad z\in\Om. 
$$
\end{proposition}
%%%%%%%%%%%%%%%%%
\begin{proof}
Let $\al\in\Om, d=d(\al)$ and consider the disk $\De =\{z : |z-\al|<d\}$. We write $z=\al+d\,\ze\in\De$, for $\ze\in\SD$, and note that $d \, \la_\De(z)=\la_\SD(\ze)$. Since $\De \subset \Om $ we have by the comparison principle \cite[Thm.8.1]{BM} that $\la_\De(z) \geq \la_\Om(z)$, for all $z\in\De$. Let $H(\ze)=h(z)$ and observe that $SH(\ze)=d^2 Sh(z)$. We have that 
$$
\|SH\|_\SD \, = \, \sup_{\ze\in\SD} \frac{|SH(\ze)|}{\la_\SD(\ze)^2}\, = \, \sup_{z\in\De} \frac{|Sh(z)|}{\la_\De(z)^2} \,  \leq \, \sup_{z\in\Om} \frac{|Sh(z)|}{\la_\Om(z)^2} \, = \, \|Sh\|_\Om. 
$$
We now apply inequality~\eqref{Pom-lem} to the function $H$ and evaluate at $\ze=0$ to obtain
$$
d(\al) \left| \frac{h''(\al)}{h'(\al)} \right| \, = \, \left| \frac{H''(0)}{H'(0)} \right| \, \leq \, 2 \sqrt{1+\frac{\|SH\|_\SD}{2}} \, \leq \, 2 \sqrt{1+\frac{\|Sh\|_\Om}{2}}. 
$$
The proof is complete. 
\end{proof}
%%%%%%%%%%%%%%%%%

%%%%%%%%%%%%%%%%%%%%%%%%%%%%%%%%%%
\subsection{The hyperbolic derivative} If $\om:\Om\to\SD$ is an analytic function then its hyperbolic derivative is given by 
$$
\om^*(z) \, =\, \frac{\om'(z)}{\la_\Om(z) \, (1-|\om(z)|^2)}, \qquad z\in\Om, 
$$
and the quantity $\|\om^*\| = \sup_{z\in\Om} |\om^*(z)|$ is called the hyperbolic norm of $\om$. In view of the generalized Schwarz-Pick lemma \cite[Thm.10.5]{BM} we always have that $\|\om^*\|\leq1$. The hyperbolic derivative satisfies the chain rule $(\om\circ\vp)^*=\om^*\circ\vp \cdot \vp^*$ for any two functions for which the composition is well defined. 

It has been shown in \cite{GZ} that for any analytic $\om:\SD\to\SD$ it holds that 
\be \label{ineq-hyp-norm}
\frac{(1-|z|^2)^2 |\om''(z)|}{1-|\om(z)|^2} \, \leq \, C \|\om^*\|, \qquad z\in\SD, 
\ee
for some constant $C>0$. We will now generalize this to an arbitrary hyperbolic domain $\Om$. Again, here $d(z)=dist(z,\partial\Om)$.

%%%%%%%%%%%%%%%%%
\begin{proposition} \label{prop-general-omega}
If $\om:\Om\to\SD$ is analytic then  
$$
\frac{d(z)^2 \, |\om''(z)|}{1-|\om(z)|^2}  \, \leq \,  C \|\om^*\|, \qquad z\in\Om,  
$$
for some constant $C>0$. 
% $$ \frac{|\om''(z)|}{1-|\om'(z)|}  \, \leq \, \frac{1}{d(z)^2} \left(2\|\om^*\|^2 + \frac{3\sqrt{3}}{2}\|\om^*\|\right), \qquad z\in\Om.  $$
\end{proposition}
%%%%%%%%%%%%%%%%%
\begin{proof}
We proceed as in the proof of Proposition~\ref{general-Pommerenke}, by fixing $\al\in\Om$ and writing $d=d(\al), \De = \{z : |z-\al|<d\}$ and $z=\al+d\,\ze\in\De$, for $\ze\in\SD$. We set $\psi(\ze)=\om(z)$ and compute  
\begin{align*}
\|\psi^*\| \, & = \, \sup_{\ze\in\SD}\frac{|\psi'(\ze)|}{\la_\SD(\ze) \, (1-|\psi(\ze)|^2)} \\ 
& = \, \sup_{z\in\De}\frac{|\om'(z)|}{\la_\De(z) \, (1-|\om(z)|^2)} \\
&  \leq \,\sup_{z\in\Om}\frac{|\om'(z)|}{\la_\Om(z) \, (1-|\om(z)|^2)} \\
& = \, \| \om^*\|. 
\end{align*}
The proof is completed by applying inequality~\eqref{ineq-hyp-norm} for $\ze=0$ \mbox{to the function $\psi$.} 
\end{proof}

%%%%%%%%%%%%%%%%%%%%%%%%%%%%%%%%%%
\subsection{Admissible dilatations} Let $\Om$ be a simply connected domain with $0\in\Om$ and let $\cA_t(\Om)$ and $\cA_t^0(\Om)$ denote the classes of admissible dilatations for mappings in $\cF_t(\Om)$ and $\cF_t^0(\Om)$, respectively. Let also 
$$
R_t(\Om) \, =\, \max_{\om\in\cA_t^0(\Om)}\|\om^*\|.  
$$
Applying the affine transformation \eqref{aff-trans} to a mapping $f\in \cF_t(\Om)$, as we have already seen, we can get a mapping $F= A_a f$ which, for an appropriate choice of $a\in\SD$, belongs to $ \in \cF^0_t(\Om)$. The dilatation of $F$ is $\om_F=\nu\vp_{\overline{a}}\circ\om$, for some $\nu\in\ST$. A straightforward computation can show that $|\om_F^*| = |\om^*|$, so that we have an alternative expression for $R_t$ given by 
\be \label{alt-R-la}
R_t(\Om) \,=\, \sup_{\om\in\cA_t(\Om)}\|\om^*\|. 
\ee
This was first observed in \cite[Lem.1]{CHM}. A compactness argument was used in \cite{HM15-2} to show that $R_t(\SD)\to0$ as $t\to0^+$. Here we will prove the following. 

\begin{lemma} \label{lem-R-lambda-Om}
It holds that $R_t(\Om)\leq 4 R_t(\SD)$ for all $t>0$. 
\end{lemma}

An immediate consequence is that $R_t(\Om)\to0$ as $t\to0^+$, and this fact is an important ingredient in the proofs of Theorems~\ref{thm-harmonic-in-rad} and \ref{thm-qc-extension}. For the proof of Lemma~\ref{lem-R-lambda-Om} we need to recall the inequalities 
\be \label{Koebe-comparison}
\frac{1}{4} \, \leq \, d(z) \la_\Om(z) \, \leq \, 1, \qquad z\in\Om, 
\ee
where $d(z)=dist(z,\partial\Om)$, which amount to Koebe's 1/4-theorem and the comparison principle; see \cite[Ch.I, \S1.1]{Leh}.

 %%%%%%%%%%%%%%%%%
\begin{proof}[Proof of Lemma \ref{lem-R-lambda-Om}]
Fix $t>0$ and consider an extremal mapping $f=h+\overline{g}$ in $\cF^0_t(\Om)$, with dilatation $\om$, for which $R_t(\Om) \,=\,\|\om^*\|$. There exists a sequence of points $\{z_n\}$ in $\Om$ for which $\|\om^*\|=\lim_{n\to\infty}|\om^*(z_n)|$. Let $r_n=d(z_n)$ and consider the disks $\De_n =\{z : |z-z_n|<r_n\}$. Hereafter we use the notation $\ze \in\SD$ and $z=z_n+r_n \ze \in \De_n$. Note that $r_n \, \la_{\De_n}(z)=\la_\SD(\ze)$ and that $\la_{\De_n}(z) \geq \la_\Om(z)$, for $z\in\De_n$, since $\De_n \subset \Om$. Let
$$
F_n(\ze) \, = \, \frac{f(z)-f(z_n)}{r_n \, h'(z_n)}
$$
and compute its dilatation $\om_n(\ze) = \mu_n \,\om(z)$, where $\mu_n=h'(z_n)/\overline{h'(z_n)}\in\ST$. Compute also $S_{F_n}(\ze)=r_n^2 S_f(z)$. We have that 
$$
 \frac{|S_{F_n}(\ze)|}{\la_\SD(\ze)^2}\, = \, \frac{|S_f(z)|}{\la_{\De_n}(z)^2} \, \leq \, \frac{|S_f(z)|}{\la_{\Om}(z)^2}, 
$$
so that 
$$
\|S_{F_n}\|_\SD \, \leq \,\|S_f\|_{\Om}  \, \leq \, t
$$
and, therefore, we find that $F_n$ belongs to $\cF_t(\SD)$. Recalling the expression \eqref{alt-R-la} we have that $\|\om_n^*\| \leq R_t(\SD)$. A calculation shows that
$$
\om_n^*(\ze) \, = \, \frac{\om_n'(\ze) }{\la_\SD(\ze)(1-|\om_n(\ze)|^2)} \, = \, \frac{\mu_n \, \om'(z) }{\la_{\De_n}(z)(1-|\om(z)|^2)} \, = \,  \mu_n \, \frac{\la_\Om(z)}{\la_{\De_n}(z)} \,\om^*(z)
$$
and, in particular, that 
$$
|\om_n^*(0)| \, = \,  r_n \, \la_\Om(z_n) \, |\om^*(z_n)|.
$$
In view of Koebe's 1/4-theorem \eqref{Koebe-comparison} we have that $ r_n \, \la_\Om(z_n) \geq 1/4$. Therefore, 
$$
|\om^*(z_n)| \, \leq \, 4|\om_n^*(0) | \, \leq \,4 \|\om_n^*\| \, \leq \, 4 R_t(\SD). 
$$
The proof is completed upon letting $n\to\infty$. 

\end{proof}
%%%%%%%%%%%%%%%%%

\section{Proof of Theorem~\ref{thm-harmonic-in-rad}}
%%%%%%%%%%%%%%%%%%%%%%%%%%%%%%%%%%%%%%%%%%%%%%%%%%
%%%%%%%%%%%%%%%%%%%%%%%%%%%%%%%%%%%%%%%%%%%%%%%%%%

%%%%%%%%%%%%%%%%%
Without loss of generality we may assume that $0\in\Om$ and that \mbox{$f \in \cF_t(\Om), t\geq0$.} By Proposition~\ref{general-Sf-iff-Sh} we have that $\|Sh\|_\Om <\infty$. Moreover, since the classes $\cF_t(\Om)$ are nested and increasing with $t$ we have that $\|Sh\|_\Om$ is uniformly bounded for $f \in \cF_t(\Om)$ and small $t$, that is, $\|Sh\|_\Om \leq M$ for some constant $M>0$ and, say, $t\leq1$. From the expression \eqref{Sch-HM} we get 
$$
|Sh| \, \leq \, |S_f| + \frac{|\om'|}{1-|\om|^2} \left| \frac{h''}{h'}\right|  + \frac{|\om''|}{1-|\om|^2} + \frac{3}{2} \left( \frac{|\om'|}{1-|\om|^2} \right)^2. 
$$
 Propositions~\ref{general-Pommerenke} and \ref{prop-general-omega}, along with Koebe's 1/4-theorem \eqref{Koebe-comparison} yield
\begin{align*}
\frac{|Sh(z)|}{\la_\Om(z)^2} \,  & \leq \, \|S_f\|_\Om +  \frac{2 |\om^*(z)| }{d(z) \la_\Om(z) } \sqrt{1+\frac{\|Sh\|_\Om}{2}}+ \frac{C\| \om^*\|}{d(z)^2 \la_\Om(z)^2 } + \frac{3}{2}|\om^*(z)|^2 \\
&\leq \, t + 8 \| \om^*\| \sqrt{1+\frac{M}{2}}   + 16C\| \om^*\| + \frac{3}{2}\| \om^*\|^2.  
\end{align*}
Since $\| \om^*\|\leq1$ we get that 
$$
\|Sh\|_\Om \,  \leq \, t + \hat{C} R_t(\Om) 
$$
for some constant $\hat{C}>0$. Hence, $\|Sh\|_\Om \to 0$ as $t\to0^+$, in view of Lemma~\ref{lem-R-lambda-Om}. Since $\si(\Om)>0$ by Theorem~\ref{thm-inner-radius-Ah-Ge}, there exists $t_0>0$ for which $\|Sh\|_\Om < \si(\Om)$, so that $h$ is univalent in $\Om$. Moreover, $h$ has a quasiconformal extension to $\overline{\SC}$ by Theorem~\ref{thm-inner-radius-Ah-Ge}, which we will denote by $\widetilde{h}$.

Using the affine invariance of $\cF_t(\Om)$, the above calculation can be repeated for the transform $A_a f$, given in \eqref{aff-trans}, for any $a\in\SD$. Thus, the analytic part of $A_a f$ and, therefore, $h_a=h+ag$ is univalent in $\Om$ for every $a\in\SD$. Letting $|a|\to1^-$ and using Hurwitz' theorem we get that $h_a$ is univalent in $\Om$ for every $a\in\ST$, since it can not be constant due to its normalization $h_a'(0)=1+ag'(0)\neq0$. We show that $f$ is injective in $\Om$ by contradiction.  Let $z_1, z_2\in \Om$ be distinct, and such that $f(z_1)=f(z_2)$. Since $h$ is injective, we have that $h(z_1)\neq h(z_2)$. Setting $\t=\arg\big(  h(z_1)- h(z_2) \big)$ we see that 
\be\label{1-1-z_12}
\SR \ni e^{-i\t} \big(h(z_1)- h(z_2)\big) =  e^{-i\t} \big(\overline{g(z_2)}- \overline{g(z_1)}\big) = e^{i\t} \big(g(z_2)- g(z_1)\big), 
\ee
from which we get that $h+e^{2i\t}g$ is not injective, a contradiction.

Moreover, $h_a$ has a quasiconformal extension $\widetilde{h_a}$ to $\overline{\SC}$ for every $a\in\SD$. In view of Theorem 5.3 in \cite[Ch.II, \S 5.4]{LV}, for every $a\in\ST$ the limit function $\widetilde{h_a}$ is either constant, or takes two values, or is quasiconformal. The first two cases are discarded by the normalization at the origin and, therefore, $\widetilde{h_a}$ is quasiconformal in $\overline{\SC}$ for every $a\in\overline{\SD}$. Now we define 
$$
\widetilde{f} =\widetilde{h}+\overline{\widetilde{g}},
$$ 
where $\widetilde{h}=\widetilde{h_0}$ and $\widetilde{g}=\widetilde{h_1}-\widetilde{h_0}$. It is clear that $\widetilde{f}$ is a continuous (with respect to the spherical metric) extension of $f$ to $\overline{\SC}$. It remains to show that $\widetilde{f}$ is injective in $\overline{\SC}$, since the continuity of $\widetilde{f}^{-1}$ may then be obtained by a general result, see Theorem~5.6 in \cite[\S3-5]{Mu}.

We will argue again by contradiction. Let $z_1, z_2\in\overline{\SC}$ be distinct, and such that $\ell=\widetilde{f}(z_1)=\widetilde{f}(z_2)\in\overline{\SC}$. Since $\widetilde{h}$ is injective, we have that $\widetilde{h}(z_1)\neq\widetilde{h}(z_2)$. If both $\widetilde{h}(z_1)$ and $\widetilde{h}(z_2)$ are finite then we proceed as in \eqref{1-1-z_12}, setting $\theta=\arg\big(  \widetilde{h}(z_1)- \widetilde{h}(z_2) \big)$ and seeing that $\widetilde{h}+e^{2i\t}\widetilde{g}$ is not injective, a contradiction. If one of $\widetilde{h}(z_1)$ and $\widetilde{h}(z_2)$ is finite and the other is infinite, we may assume without loss of generality that $\widetilde{h}(z_1)=\infty$ and $\widetilde{h}(z_2)\neq\infty$. Let $\ga_\t$ be the pre-image of the ray $\{Re^{i\t} : R\geq0\}$, $\t\in[0,2\pi)$, of the function $\widetilde{h}$. Clearly, $\ga_\t$ is a simple curve with one endpoint at the origin and the other at $z_1$. We write 
$$
\widetilde{h}+\widetilde{g} = \overline{\widetilde{f}} + \widetilde{h} - \overline{\widetilde{h}}
$$ 
and see that for $z\in\ga_\t$ we have that $\widetilde{h}(z)-\overline{\widetilde{h}}(z) = R e^{i\t}(1-e^{-2i\t})$. We distinguish two cases: $\ell$ (the common value of $\widetilde{f}$ at $z_1$ and $z_2$) being finite or infinite. If $\ell \neq \infty$ then 
$$
\lim_{\ga_\t \ni z\to z_1}\widetilde{h}(z)+\widetilde{g}(z) = 
\left\{
\begin{array}{rl}
\overline{\ell}, & \quad \text{if} \quad \t=0, \pi,\\
 \infty, & \quad \text{if} \quad \t\neq 0,\pi, 
\end{array} \right.
$$
which is a contradiction since $ \widetilde{h}+\widetilde{g} $ is continuous in $\overline{\SC}$. If $\ell = \infty$ we take limit when $z\to z_1$ along $\ga_0$ (so that $ \widetilde{h}(z)=R$) in order to compute that $(\widetilde{h}+\widetilde{g})(z_1)  = \infty$. But, on the other hand, we have that $(\widetilde{h}+\widetilde{g})(z_2)= \infty$, since $ \widetilde{h}(z_2) \neq \infty$ and $\widetilde{g}(z_2) = \infty$. This is a contradiction because $\widetilde{h}+\widetilde{g}$ is injective.

\section{Proof of Theorem~\ref{thm-qc-extension}} \label{sect-qc-ext}
%%%%%%%%%%%%%%%%%%%%%%%%%%%%%%%%%%%%%%%%%%%%%%%%%%
%%%%%%%%%%%%%%%%%%%%%%%%%%%%%%%%%%%%%%%%%%%%%%%%%%
 
 Let $\vp:\SD\to\Om$ be a Riemann map of $\Om$ with $\vp(0)=0$ and consider $\Om_r = \vp(\{|z|<r\})$ for $r<1$. We set $\ga_r = \partial \Om_r$ and note that, since $\Om$ is a $\widehat{K}$-quasidisk for some $\widehat{K}\geq1$, it follows (trivially) that $\ga_r$ is a $\widehat{K}$-quasicircle for every $r<1$. We also set $\Gamma_r = f(\ga_r)$ and claim that it is a $K$-quasicircle, with $K\geq1$ independent of $r$. Once we prove this claim we may then consider $\la_r$ and $\Lambda_r$ to be $\widehat{K}$- and $K$-quasiconformal reflecions across $\ga_r $ and $\Gamma_r$, respectively, and setting 
$$
\widetilde{f_r}(z)  \,=\, \left\{
\begin{array}{rl}
f(z),   & \quad \text{if} \quad z\in\overline{\Om_r},\\
\Lambda_r \circ f \circ \la_r(z) , & \quad \text{if} \quad z\in\overline{\SC}\backslash\overline{\Om_r}, 	
\end{array} \right.
$$
we see that this is a $\le(K\,\frac{1+d}{1-d}\,\widehat{K}\ri)$-quasiconformal mapping in the Riemann sphere. Letting $r\to1$, and in view of Theorem 5.3 in \cite[Ch.II, \S 5.4]{LV}, we obtain the desired quasiconformal extension of $f$ to $\overline{\SC}$.

Let $w_j\in\Gamma_r, j=1,2,3,4$, be distinct points. Our claim that $\Gamma_r$ is a $K$-quasicircle, with $K\geq1$ independent of $r$, will be proved by showing that the cross-ratio of the points $w_j$ is bounded by a uniform constant, independent of $r$. Since $f$ is injective in $\Om$ by Theorem~\ref{thm-harmonic-in-rad}, there exist exactly four points $z_j\in\ga_r$ for which $w_j=f(z_j)$. For convenience, we will write $h_j=h(z_j)$ and $g_j=g(z_j)$. We have
$$
w_i-w_j = h_i-h_j  + \overline{g_i}- \overline{g_j} =  (h_i-h_j) (1+\mu_{ij}A_{ij}), 
$$ 
where $A_{ij}= \frac{ g_i-g_j }{ h_i-h_j }$ and $\mu_{ij}=\frac{\overline{g_i}- \overline{g_j}}{g_i-g_j } \in \ST$ (we may set $\mu_{ij}=1$ if $g_i=g_j$). We have that 
\be\label{cross-proof}
|(w_1,w_2,w_3,w_4)| = |(h_1,h_2,h_3,h_4)| \le| \frac{ (1+\mu_{12}A_{12}) (1+\mu_{34}A_{34}) }{ (1+\mu_{13}A_{13}) (1+\mu_{24}A_{24}) } \ri|. 
\ee
Since $h(\Om)$ is a quasidisk (in view of the proof of Theorem~\ref{thm-harmonic-in-rad}), we have that 
$$
|(h_1,h_2,h_3,h_4)|\leq M
$$
for some absolute constant $M\geq0$. 

We saw in the course of the proof of Theorem~\ref{thm-harmonic-in-rad} that $h_a=h+ag$ is univalent in $\Om$ for every $a\in\overline{\SD}$. We will now expand the range of $|a|$ for which this holds. 

Let $\de\in(1,1/d)$ and consider $1<|a|\leq\de$. We compute $h_a' = h'(1+a\om)$ and 
$$
\frac{h_a''}{h_a'} = \frac{h''}{h'} + \frac{a\om'}{1+a\om},
$$
so that
$$
Sh_a = Sh - \frac{a\om'}{1+a\om} \frac{h''}{h'}+ \frac{a\om''}{1+a\om} - \frac{3}{2}\le(\frac{a\om'}{1+a\om}\ri)^2.
$$
By formula \eqref{Sch-HM} and a straightforward computation we arrive at
$$
Sh_a = S_f + \frac{a+\overline{\om}}{1+a\om} \le[ \frac{\om''}{1-|\om|^2} - \frac{\om'}{1-|\om|^2} \frac{h''}{h'} + \frac{3}{2} \le(\frac{ \om'}{1-|\om|^2}\ri)^2\le(\overline{\om} - \frac{a(1-|\om|^2)}{1+a\om} \ri) \ri]. 
$$
In $\Om$, we have that 
$$
\le|\frac{a+\overline{\om}}{1+a\om}\ri| \leq \max_{|\ze|= d} \le|\frac{a+\overline{\ze}}{1+a\ze}\ri|  = \frac{|a|-d}{1-|a|d} \leq \frac{\de-d}{1-\de d}=C,  
$$
where the first of these inequalities follows from the maximum principle. Moreover, we have that
$$
\le|\overline{\om} - \frac{a(1-|\om|^2)}{1+a\om} \ri| \leq d + \frac{|a|(1-|\om|^2)}{1-|a\om|} \leq d + \frac{|a|(1-d^2)}{1-|a|d}\leq d + \frac{\de(1-d^2)}{1-\de d}=C'. 
$$
Hence, we get that  
$$
|Sh_a| \leq |S_f| + C \le[ \frac{|\om''|}{1-|\om|^2} + \frac{|\om'|}{1-|\om|^2} \le|\frac{h''}{h'}\ri| + \frac{3C' }{2} \le(\frac{ |\om'|}{1-|\om|^2}\ri)^2  \ri]. 
$$
We assume that $f \in \cF_t(\Om)$, for $t\geq0$, and, working as in the proof of Theorem~\ref{thm-harmonic-in-rad}, we use Propositions~\ref{general-Pommerenke}, \ref{prop-general-omega} and Koebe's 1/4-theorem \eqref{Koebe-comparison} in order to obtain
$$
\|Sh_a\|_\Om \,  \leq \, t + \hat{C} R_t(\Om) 
$$
for some constant $\hat{C}>0$. Choosing $t_0>0$ so that $t_0 + \hat{C} R_{t_0}(\Om) = \si(\Om)$ we conclude that $h_a$ is univalent.

We fix $\al\in\Om$ and consider the generalized dilatation 
$$
\psi_\al(z)  \,=\, \left\{
\begin{array}{rl}
\frac{g(z)-g(\al)}{h(z)-h(\al)},   & \quad \text{if} \quad z\in\Om\backslash \{\al\},\\
\om(\al), & \quad \text{if} \quad z=\al.  
\end{array} \right.
$$
Since $h$ is injective in $\Om$ it is clear that $\psi_\al$ is holomorphic in $\Om$. We claim that 
$$
S=\max_{z\in\overline{\Om_r}} |\psi_\al(z) | \leq \frac{1}{\de}. 
$$
Note that $|\om(\al)|\leq d<1/\de$ so that if, otherwise, $S>1/\de$, then there would exist a point $z_0\in\Om\backslash\{\al\}$ for which $\psi_\al(z_0)=e^{i\t}/\de$, for some $\t\in\SR$. This shows that the values of the function $h-e^{-i\t}\de g$ at the points $\al$ and $z_0$ would coincide, which is a contradiction since this is a univalent function. 

Hence we may bound the terms in \eqref{cross-proof} as $|A_{ij}|\leq 1/\de$ in order to obtain 
$$
|(w_1,w_2,w_3,w_4)| \leq M  \frac{ (1+|A_{12}|) (1+|A_{34}|) }{ (1-|A_{13}|) (1-|A_{24}|) } \leq M \le( \frac{1+1/\de}{1-1/\de}\ri)^2, 
$$
with which we finish the proof.

\section{Finitely connected domains} \label{sect-fin-conn}  
%%%%%%%%%%%%%%%%%%%%%%%%%%%%%%%%%%%%%%%%%%%%%%%%%%
%%%%%%%%%%%%%%%%%%%%%%%%%%%%%%%%%%%%%%%%%%%%%%%%%%

Let $\Om$ be a domain in $\overline{\SC}$. A collection $\mathfrak{D}$ of domains $D\subset\Om$ is called a quasiconformal decomposition of $\Om$ if each $D$ is a quasidisk and any two points $z_1,z_2\in \Om$ lie in the closure of some $D\in\mathfrak{D}$. This definition along with the following covering lemma were given by Osgood in \cite{O80}.

\begin{lemmO}[\cite{O80}] \label{lem-qc-decomp}
If $\Om$ is a finitely connected domain and each component of $\partial \Om$ is either a point or a quasicircle then $\Om$ is quasiconformally decomposable. 
\end{lemmO}

The proof in \cite{O80} provides an explicit finite decomposition.

\begin{proof}[Proof of Therorem~\ref{thm-finit-conn}]
We prove the direction (i)$\Rightarrow$(iii). The domain $\Om$ is quasiconformally decomposable by a collection $\mathfrak{D}$ in view of Lemma~\ref{lem-qc-decomp}. By Theorem~\ref{thm-harmonic-in-rad}, each of the quasidisks $D$ in $\mathfrak{D}$ has positive harmonic inner radius. Let
$$
0<c\leq \min_{D\in\mathfrak{D}} \si_H(D)
$$ 
and consider $f$ to be a harmonic mapping in $\Om$ that satisfies $\|S_f\|_\Om\leq c$. If $f(z_1)=f(z_2)$ for two distinct points $z_1,z_2$ in $\Om$ then $z_1,z_2\in\overline{D}$, for some quasidisk $D$ from the collection $\mathfrak{D}$. The domain monotonicity for the hyperbolic metric shows that $\la_D(z)\geq \la_\Om(z)$ for all $z\in D$ and, therefore, that 
$$
\|S_f\|_D \leq \|S_f\|_\Om\leq c. 
$$ 
But now the homeomorphic extension of Theorem~\ref{thm-harmonic-in-rad} shows that if $c$ is sufficiently small then $f$ is injective up to the boundary of $D$, a contradiction. 
\end{proof}

We note that if we strengthen the definition of quasiconformal decomposition so that any two points $z_1,z_2\in \Om$ lie in some quasidisk $D$ (not its closure) from a collection $\mathfrak{D}$, then the construction in \cite{O80} can be modified so that Lemma~\ref{lem-qc-decomp} still holds. Had we followed this line of reasoning then we would not need the homeomorphic extension of Theorem~\ref{thm-harmonic-in-rad}, but only the univalence criterion from its statement.

\vskip.3cm
%%%%%%%%%%%%%%%%%%%%%%%%%%%%%%%%%%%%%
%%%%%%%%%%%%%%%%%%%%%%%%%%%%%%%%%%%%%
\emph{Acknowledgements}. I would like to thank professor Martin Chuaqui for his insightful comments in our many discussions.

%%%%%%%%%%%%%%%%%%%%%%%%%%%%%%%%%%%%%
%%%%%%%%%%%%%%%%%%%%%%%%%%%%%%%%%%%%%

\begin{thebibliography}{99}
%%%%%%%%%%%%%%%%%
\bibitem{Ah} L. Ahlfors, Quasiconformal reflections, \emph{Acta Math.} \textbf{109} (1963), 291-301.  

\bibitem{BG} A.F. Beardon, F.W. Gehring, Schwarzian derivatives, the Poincar\'e metric and the kernel function, \textit{Comment. Math. Helv.} \textbf{55} (1980), no.~1, 50-64.

\bibitem{BM} A. Beardon, D. Minda, The hyperbolic metric and geometric function theory, \emph{Quasiconformal mappings and their applications}, 9-56, Narosa, New Delhi, 2007.

\bibitem{CDO03} M. Chuaqui, P. Duren, B. Osgood, The Schwarzian derivative for harmonic mappings, \textit{J. Anal. Math.} \textbf{91} (2003), 329-351. 

\bibitem{CHM} M. Chuaqui, R. Hernández, M.J. Mart\'in, Affine and linear invariant families of harmonic mappings, \emph{Math. Ann.} \textbf{367} (2017), no.~3-4, 1099-1122.

\bibitem{Du3} P.L. Duren, \emph{Harmonic Mappings in the Plane}, Cambridge University Press, Cambridge, 2004.

\bibitem{Ge} F.W. Gehring, Univalent functions and the Schwarzian derivative, \emph{Comment. Math. Helv.} \textbf{52} (1977), no.~4, 561-572. 

\bibitem{GeHa} F.W. Gehring, K. Hag, \emph{The ubiquitous quasidisk}, Amer. Math. Soc., Providence, RI, 2012. 

\bibitem{GZ} P. Ghatage, D. Zheng, Hyperbolic derivatives and generalized Schwarz-Pick estimates, \emph{Proc. Amer. Math. Soc.} \textbf{132} (2004), no.~11, 3309-3318. 

\bibitem{HM13} R. Hern\'andez, M.J. Mart\'in, Quasiconformal extension of harmonic mappings in the plane, \textit{Ann. Acad. Sci. Fenn. Math.} \textbf{38} (2013), no.~2, 617-630. 

\bibitem{HM15} R. Hern\'andez, M.J. Mart\'in, Pre-Schwarzian and Schwarzian derivatives of harmonic mappings, \textit{J. Geom. Anal.} \textbf{25} (2015), no.~1, 64-91. 

\bibitem{HM15-2} R. Hern\'andez, M.J. Mart\'in, Criteria for univalence and quasiconformal extension of harmonic mappings in terms of the Schwarzian derivative, \textit{Arch. Math. (Basel)} \textbf{104} (2015), no.~1, 53-59. 

\bibitem{Ku88} R. K\"uhnau, M\"oglichst konforme Spiegelung an einer Jordankurve, \textit{Jahresber. Deutsch. Math.-Verein.} \textbf{90} (1988), no. 2, 90-109.

\bibitem{Leh} O. Lehto, \textit{Univalent functions and Teichm\"uller spaces}, Graduate Texts in Mathematics, vol.~109, Springer-Verlag, New York, 1987.  

\bibitem{LV} O. Lehto, K.I. Virtanen, \emph{Quasiconformal mappings in the plane}, second edition, Springer-Verlag, New York-Heidelberg, 1973.  

\bibitem{MS79} O. Martio, J. Sarvas, Injectivity theorems in plane and space, \emph{Ann. Acad. Sci. Fenn. Ser. A I Math.} \textbf{4} (1979), no. 2, 383-401. 

\bibitem{Mu} J.R. Munkres, \emph{Topology, a first course}, Prentice-Hall, Inc., Englewood Cliffs, N.J., 1975. 

\bibitem{O80} B. Osgood, Univalence criteria in multiply-connected domains, \emph{Trans. Amer. Math. Soc.} \textbf{260} (1980), no.~2, 459-473. 

\bibitem{Po64} Ch. Pommerenke, Linear-invariante Familien analytischer Funktionen I, \emph{Math. Ann.} \textbf{155} (1964), no.~2, 108-154. 
\end{thebibliography}
\end{document}